\documentclass[final,1p,times]{elsarticle}

\usepackage{amssymb}
\usepackage{latexsym}
\usepackage{amsfonts}
\usepackage{amsmath}
\usepackage{eucal}
\usepackage{amscd}
\usepackage{graphics}
\usepackage{epsfig}
\usepackage{fancyhdr}
\usepackage{verbatim}
\usepackage{setspace}
\usepackage[all]{xy}
\usepackage{pb-diagram}
\usepackage{amsthm}

\newcommand{\Hh}{\mathcal{H}}
\newcommand{\ep}{\varepsilon}
\newcommand{\R}{\mathbb{R}}
\newcommand{\N}{\mathbb{N}}

\newtheorem{thm}{Theorem}
\newtheorem{lem}[thm]{Lemma}
\newtheorem{cor}[thm]{Corollary}
\newtheorem{remark}{Remark}

\journal{}

\begin{document}

\begin{frontmatter}



\title{The number of medium amplitude limit cycles of some generalized Li\'enard systems}


\author{Salom\'on Rebollo-Perdomo}
\ead{srebollo@mat.uab.cat}
\address{Centre de Recerca Matem\`atica, 08193 Bellaterra,
Barcelona, Spain.
}

\begin{abstract}
We will consider two special families of  polynomial perturbations of the
linear center. For the resulting perturbed systems, which are generalized
Li\'enard systems, we provide the exact upper bound for the number of limit
cycles that bifurcate from the periodic orbits of the linear center.
\end{abstract}


\end{frontmatter}


\section{Introduction and statement of the results}
\label{Intro}
The bifurcation of limit cycles by perturbing a planar system which has
a continuous family of {\it cycles}, {i.e.} periodic orbits, has been an
intensively studied phenomenon; see for instance \cite{CL} and references
therein.
The simplest planar system having a continuous family of cycles is the
linear center, and a special family of its perturbations is
given by the generalized polynomial Li\'enard systems:
\begin{equation}
\label{eq1}
\tag{$1_\ep$}
\dot{x} = y + \sum_{i=1}^{\mu}\ep^i F_i(x),\qquad
\dot{y} = \sum_{i=0}^{\nu}\ep^i g_i(x),
\end{equation}
where $\mu\in \N$, $\nu\in \N \cup \{0\}$, $g_0(x)=-x$, $g_i(x)$ and $F_i(x)$
 are polynomials for $i\geq 1$, and $\ep$ is a small parameter.

The classical and generalized Li\'enard systems appear very often in several
branches of science and
engineering, as biology, chemistry, mechanics, electronics,
etc. see for instance \cite{M} and references therein. In particular Li\'enard
systems are frequent specially in physiological
processes, see for instance \cite{GFD}. In addition, the family of generalized
polynomial Li\'enard systems is one of the most considered families in the
study of limit cycles, see \cite{L}.

We assume that $F_{\mu}(x)\not \equiv 0$, $g_{\nu}(x)\not \equiv 0$,
$m=\max_{1\leq i\leq \mu}\{\deg F_i(x)\}$, and
$n=\max_{1\leq i\leq \nu}\{\deg g_i(x)\}$. For a small enough $\ep$, let
$\Hh_{\nu}^{\mu}(m,n)$  be the maximum number of limit
cycles of \eqref{eq1} that bifurcate from cycles of the {\it linear center}
$(1_0)$, i.e. the maximum number of {\it medium amplitude limit cycles}
which can bifurcate from  $(1_0)$ under the perturbation \eqref{eq1}. If
$\nu=0$, then $\Hh_{0}^{\mu}(m,n)$ does not depend on $n$; hence we only
write $\Hh_{0}^{\mu}(m)$.
The main problem concerning $\Hh_{\nu}^{\mu}(m,n)$ is
finding its exact value.

We know from \cite{LMP} that $\Hh_0^{1}(m)\geq
\left[(m-1)/2\right]$, where $[\cdot]$ denotes the integer part function.
Moreover, by following \cite[Theorem 3.1]{GT2} we can prove
that  $\Hh_0^{\mu}(m)=\left[(m-1)/2\right]$ for $\mu\geq
1$; Theorem~\ref{mth1} (below)
is a generalization of this result. Also,
we know from \cite{LMT} that $\Hh_1^1(m,n)\geq\left[(m-1)/2\right]$,  $\Hh_2^2(m,n)\geq
\max\left\{\left[(m-1)/2\right],
\left[m/2\right]+\left[n/2\right]-1\right\}$, and $\Hh_3^3(m,n)\geq
\left[(m+n)/2\right]-1$. However, the exact values of
$\Hh_1^1(m,n)$,  $\Hh_2^2(m,n)$, and $\Hh_3^3(m,n)$
were not reported there.

In this paper we  give the exact value of $\Hh_{\nu}^{\mu}(m,n)$ for two
subfamilies of
\eqref{eq1}. More precisely, we will give the exact value of $\tilde{\Hh}_{\nu}^{\mu}(m,n)$ and $\bar{\Hh}_{\nu}^{\mu}(m,n)$, where  $\tilde{\Hh}_{\nu}^{\mu}(m,n)$ is  the
 value of $\Hh_{\nu}^{\mu}(m,n)$ by assuming that $g_i(x)$  is odd for
$1\leq i\leq \nu$, and  $\bar{\Hh}_{\nu}^{\mu}(m,n)$ is the
 value of $\Hh_{\nu}^{\mu}(m,n)$ by assuming that $F_i(x)$ is even for
$\mu_0 < i\leq \mu$, where $\mu_0$  with $1\leq \mu_0\leq \mu$ is the smallest
integer such
that $F_{\mu_0}(x)\not \equiv 0$. Of course, if $\mu_0=\mu$, then $\bar{\Hh}_{\nu}^{\mu}(m,n)=\Hh_{\nu}^{\mu}(m,n)$.

Our main result is the following:

\begin{thm}
\label{mth1}
$(a)$
$\tilde{\Hh}_{\nu}^{\mu}(m,n)=\left[\frac{m-1}{2}\right]$. $(b)$
$\bar{\Hh}_{\nu}^{\mu}(m,n)$ is either $\left[\frac{m-1}{2}\right]$ if $m$ is
odd or
$\left[\frac{m}{2}\right]+ \left[\frac{n}{2}\right]-1$
if $m$ is even.
\end{thm}

The assumptions on $g_i(x)$ and $F_i(x)$ in definitions of
$\tilde{\Hh}_{\nu}^{\mu}(m,n)$ and
$\bar{\Hh}_{\nu}^{\mu}(m,n)$, respectively,  are necessary. Otherwise, we can construct systems  \eqref{eq1} having more medium amplitude limit cycles, see Remark 1 in Section~\ref{sec3}.

Theorem \ref{mth1}($b$) is a generalization of Theorem 1.1 in \cite{RP}, where
the case $\mu=\nu=1$ was considered. We note that in such a  case
$\bar{\Hh}_{1}^{1}(m,n)={\Hh}_{1}^{1}(m,n)$. Hence Theorem 1$(b)$ (\cite[Theorem 1.1]{RP}) gives the exact value of ${\Hh}_{1}^{1}(m,n)$.

The proof of Theorem \ref{mth1}  is based on  computing
the maximum number of  isolated zeros
of the first non-vanishing Poincar\'e--Pontryagin--Melnikov function
of
the displacement function of \eqref{eq1}, by taking into account the
restrictions: $g_i(x)$ odd for
$1\leq i\leq \nu$ and $F_i(x)$ even for $\mu_0< i\leq \mu$, respectively.

The paper is organized as follows. In Section \ref{PPMF} we recall the
definition of the displacement function of \eqref{eq1}, as well as the
algorithm to compute the  Poincar\'e--Pontryagin--Melnikov functions.
Preliminary results that allow us to provide elementary proofs
of the main results are given in Section~\ref{sec3}. Finally, in
Section~\ref{sec4} we will  prove  Theorem \ref{mth1}.

\section{Poincar\'e--Pontryagin--Melnikov functions}
\label{PPMF}
The
linear center $(1_0)$ is the Hamiltonian system associated to the  polynomial
$H=(x^2+y^2)/2$; hence its cycles are the circles
$\gamma_c= \{H-c=0\}$  with
$c>0$. By using $c$ as a parameter, the first return map of
\eqref{eq1}  can be
expressed in terms of $\ep$ and
$c$: $\mathcal{P}(\ep,c)$. Therefore the corresponding {\it displacement
function}
$L(\ep,c)=
\mathcal{P}(\ep,c) - c$ is analytic for  small enough $\ep$ and can be
written as the power series in
$\ep$
\begin{equation}
\label{df}
\tag{$2$}
 L(\ep,c)= \ep L_1(c)+\ep^2 L_{2}(c) +
O(\ep^{3}),
\end{equation}
where $L_i(c)$ with $i\geq 1$ is the
{\it Poincar\'e--Pontryagin--Melnikov function} of order
$i$, which is defined for $c\geq 0$.

Let $L_k(c)$ with $k\geq 1$ be the first non-vanishing
coefficient in \eqref{df}. The zeros of $L_k(c)$ are
important in the study of medium amplitude limit cycles of
\eqref{eq1} because of the
{\it Poincar\'e--Pontryagin--Andronov criterion}: The maximum number of
isolated zeros, counting multiplicities, of $L_k(c)$
is an upper bound
for $\Hh_{\nu}^{\mu}(m,n)$. Furthermore each
simple zero $c_0$ of $L_k(c)$
corresponds
to  one and only one limit cycle of \eqref{eq1} with $\ep$ small enough
bifurcating from the cycle
$\gamma_{c_0}$.

We know from \cite{Ili} that $L_k(c)$  has at most
$\left[k(\max\{n,m\}-1)/2\right]$ positive zeros, counting multiplicities. However,
this result  does not
give the value of $\Hh_{\nu}^{\mu}(m,n)$ because
the upper bound for $k$ depending on $\mu$, $\nu$, $m$, and $n$ is unknown.

Now, we will recall the algorithm to compute the functions $L_i(c)$.  System
\eqref{eq1}
can
be written as
$$
\dot{x}=y,\qquad
\dot{y}=-x+ \ep \left(g_1(x)+f_1(x) y\right)+\ep^2 \left(g_2(x)+f_2(x)
y\right)+\cdots,
$$
where $f_i(x)=F_i^{\prime}(x)$, or equivalently as
\begin{equation}
\label{Ll}
\tag{$3_\ep$}
\begin{array}{rl}
dH -\ep \omega_1-\ep^2 \omega_2 -\cdots =0\qquad \mbox{with
$\omega_i=\left(g_i(x)+f_i(x)y\right)dx$ and $\omega_i\equiv 0$ for $i\geq
\max\{\mu,\nu\}$.}
\end{array}
\end{equation}

As we know, $L_1(c)$ is given by  the classical
Poincar\'e--Pontryagin
formula
$
L_1(c) = \int_{\gamma_c}\omega_1.
$
The  result for computing the higher order
Poincar\'e--Pontryagin--Melnikov
functions is the following:
\begin{thm}
\label{th3}
{\rm \bfseries (Yakovenko--Fran\c{c}oise--Iliev algorithm \cite{Ya},
\cite{Fr}, \cite{Ili})}. If $k\geq 2$ and  $L_1(c)\equiv\cdots\equiv
L_{k-1}(c)\equiv 0$,
then there are
polynomials $q_1,\ldots,q_{k-1}$ and $Q_1,\ldots,Q_{k-1}$ such that
$\Omega_1=dQ_1+q_1
dH,\ldots,\Omega_{k-1}=dQ_{k-1} +q_{k-1} dH$, and
$$
L_{k}(c) =  \int_{\gamma_c} \Omega_{k},
$$
where
$$
\Omega_1=\omega_1, \,\, \mbox{and}\,\,
\Omega_l=\omega_l+\sum_{i+j=l}q_i \omega_j\,\, \mbox{with} \,\,
i,j\geq 1\,\, \mbox{for}\,\, 2\leq l\leq k.
$$
\end{thm}
The proof of this result
easily follows from the Poincar\'e--Pontryagin
formula, and the Ilyashenko--Gavrilov theorem (\cite{Ily}, \cite{Ga}):
If $\int_{\gamma_c}\omega = 0$ for all $c\geq 0$, then $\omega  =  dQ + q dH$,
where  $Q$
and
$q$ are polynomials, and by applying an induction argument. For a
detailed proof, see for instance
\cite{Ili},  \cite{IY}.

To simplify the
computation of the Poincar\'e--Pontryagin--Melnikov
functions,  we will give some  properties
of $\omega_i$.

\section{Preliminary results}
\label{sec3}
For computing $L_{k}(c)$  for \eqref{eq1} we will use the following
two elementary
lemmas
whose proof
is omitted.
\begin{lem}
\label{lem1}
Let $P$ be a polynomial in the ring $\R\left[x^2,H\right]$. We define
$\deg_2 P$ to be the
degree of $P$ in  $\R[x^2,H]$.
\begin{itemize}
\item[$(a)$]
For $i,j \geq 0$ there are homogeneous polynomials $Q_{ij},\, q_{ij} \in
\R\left[x^2,H\right]$ with $\deg_2 Q_{ij}=i+j$ and
$\deg_2 q_{ij}=i+j-1$, such that
$H^i x^{2j} dx = d\left(xQ_{ij}\right)+\left(x q_{ij}\right)dH$ or $H^i
x^{2j+1} dx =d\left(x^2Q_{ij}\right)+\left(x^2 q_{ij}\right)dH$.
If $i=0$, then $q_{ij}\equiv 0$.

\item[$(b)$]
For $i, j \geq 0$ there are  homogeneous polynomials $Q_{ij}, \,
q_{ij} \in \R[x^2,H]$ with $\deg_2 Q_{ij}=i+j+1$ and
$\deg_2 q_{ij}=i+j$, such that
$H^ix^{2j+1}ydx=d\left(yQ_{ij}\right)+\left(y q_{ij}\right)dH$.
\item[$(c)$] For $i, j \geq 0$ we have $\displaystyle
\int_{\gamma_c}H^ix^{2j}ydx =\frac{-\pi c}{2^j(2j+1)}{2(j+1) \choose
j+1} c^{i+j}.
$
\end{itemize}
\end{lem}

\begin{lem}
\label{lem2}
 If $\omega
\in \mathcal{A}:=\left\{\left(xA+xyB\right)dx \, \big{|}\, A, B
\in\R\left[x^2,H\right]\right\}$ and
$q
\in \mathcal{S}:=\left\{x^2q_1+yq_2 \,\big{|}\, q_1, q_2
\in\R\left[x^2,H\right]\right\}$, then $q \omega \in \mathcal{A}$.
\end{lem}
The next two results are straightforward consequences of these two previous
lemmas.
\begin{cor}
\label{cor1}
If $\omega
\in \mathcal{A}$, then $\int_{\gamma_c} \omega
\equiv 0$, $\omega=dQ+qdH$ with $q \in \mathcal{S}$, and $q \omega \in
\mathcal{A}$.
\end{cor}
\begin{cor}
\label{cor2}
If
$\displaystyle P{\left(x^2\right)}=\sum_{r=0}^{d}p_{r}x^{2r} \in
\R\left[x^2\right]$, then $\displaystyle \int_{\gamma_c} P{\left(x^2\right)} y
dx=-\pi c \left(
\sum_{r=0}^{d}{2(r+1) \choose r+1}\frac{
p_{r}}{2^{r}(2r+1)}{c^{r}}\right).$
\end{cor}
The following two lemmas will be important in the proof of Theorem \ref{mth1}.
\begin{lem}
\label{lem3}
Suppose Theorem \ref{th3}.  If
 $\Omega_l \in \mathcal{A}$ for $1\leq l\leq k-1$, then $\omega_l
\in
\mathcal{A}$ for $1\leq l\leq k-1$, and
$
L_{k}(c) =  \int_{\gamma_c} \omega_{k}.
$
\end{lem}
\begin{proof} We proceed by induction on $k$.
If $k=2$, then $\Omega_1=\omega_1 \in
\mathcal{A}$. Hence $\omega_1=dQ_1+q_1dH$,
$q_1\omega_1 \in \mathcal{A}$, and $\int_{\gamma_c}
q_1\omega_1\equiv 0$ by Corollary \ref{cor1}. Since
$L_2(c)=\int_{\gamma_c}\Omega_2= \int_{\gamma_c}\omega_2
+\int_{\gamma_c}q_1\omega_1$
because of Theorem \ref{th3},
$L_2(c)=\int_{\gamma_c}\omega_2$.

We assume  that the lemma is true for $k-1$, and we will prove it for
$k$.
By assumption, $\Omega_l \in \mathcal{A}$ for  $1\leq l \leq k-1$. Then,
by Corollary \ref{cor1}, $\Omega_l=dQ_l+q_ldH$
with $q_l\in \mathcal{S}$ for all $1\leq l \leq k-1$. In addition, by the
induction
hypothesis,
$\omega_l \in \mathcal{A}$ for $1\leq l \leq k-2$. Thus,
$\overline{\Omega}_{k-1}:=\sum_{i+j=k-1}q_i \omega_j$
with $i,\, j\geq 1$
is an element of $\mathcal{A}$ following Lemma~\ref{lem2}.  Since
$\omega_{k-1}= \Omega_{k-1} -
\overline{\Omega}_{k-1}$, $\omega_{k-1} \in \mathcal{A}$. Hence it is clear that
$\overline{\Omega}_{k}:=\sum_{i+j=k}q_i \omega_j$ with $i,\, j\geq 1$
is an element of $\mathcal{A}$, which implies that
$
\int_{\gamma_c}
\overline{\Omega}_{k} \equiv 0$ by  Corollary~\ref{cor1}. Finally, from Theorem
\ref{th3} we have
$L_{k}(c)=\int_{\gamma_c} \omega_{k} + \int_{\gamma_c}
\overline{\Omega}_{k}$. Therefore
$L_{k}(c)=\int_{\gamma_c}\omega_{k}$.
\end{proof}

Before announce next lemma, we note that each polynomial
$h(x)=\sum_{r=0}^{m-1}a_{r}x^{r}$ of degree $m-1$ can be written as
\begin{equation}
\label{eq7}
\tag{$4$}
h(x)=\hat{h}{\left(x^2\right)} + x \tilde{h}{\left(x^2\right)},\quad \mbox{where} \quad
\hat{h}{\left(x^2\right)}=
\sum_{r=0}^{\left[\frac{m-1}{2}\right]}a_{2r+1}x^{2r}, \quad \mbox{and} \quad
\tilde{h}{\left(x^{2}\right)}=
\sum_{r=0}^{\left[\frac{m-2}{2}\right]}a_{2r+2}x^{2r}.
\end{equation}

\begin{lem}
\label{lem4}
Let $\omega=\left(g(x)+f(x)y\right)dx$, where $f(x)=\sum_{r=0}^{m-1}a_{r}x^{r}$
and $g(x)=\sum_{s=0}^{n}b_{s}x^{s}$.

\begin{itemize}
\item[$(a)$]
$\displaystyle
\int_{\gamma_c}  \omega=\int_{\gamma_c}
\hat{f}{\left(x^2\right)} y dx=-\pi c\left(
\sum_{r=0}^{\left[\frac{m-1}{2}\right]}{2(r+1) \choose r+1}\frac{
a_{2r+1}}{2^{r}(2r+1)}{c^{r}}\right).
$
\item[$(b)$]If
$\int_{\gamma_c} \omega \equiv 0$,  then $\omega=dQ+(y\bar{q}) dH$ with $\bar{q}
\in
\R\left[x^2,H\right]$ of degree $\deg_2 \bar{q}=\left[(m-2)/2\right]$, and
$$
\int_{\gamma_c} \left(y\bar{q}\right)\omega= \int_{\gamma_c} \bar{q}
\hat{g}{\left(x^2\right)} ydx= -\pi
c\sum_{s=0}^{\left[\frac{n}{2}\right]}
\left(\sum_{r=0}^{\left[\frac{m-2}{2}\right]}{2(s+r+1) \choose s+r+1}
\frac{{\left(b_{2s}\right)}(
a_{2r+2})}{2^{s+r}(2s+1)}{c^{s+r}} \right).
$$
\item[$(c)$]$\int_{\gamma_c}
\left(y\bar{q}\right)\omega \equiv 0$ if and only if $\bar{q} \equiv 0$, or
$\hat{g}{\left(x^2\right)}\equiv 0$.
\end{itemize}
\end{lem}
\begin{proof} ($a$). By ($a$) and ($b$) of Lemma \ref{lem1},
$\int_{\gamma_c}\omega=
\int_{\gamma_c}\hat{f}{\left(x^2\right)}ydx$. Finally, the statement follows
from Corollary~\ref{cor2}.

($b$). If $\int_{\gamma_c}\omega\equiv 0$, then $\hat{f}{\left(x^2\right)}\equiv
0$ by statement ($a$). This property implies that
$\omega=g(x)
dx+x\tilde{f}{\left(x^2\right)}ydx = d\left(\int
g(x)dx\right) +
\sum_{r=0}^{\left[\frac{m-2}{2}\right]}a_{2r+2}x^{2r+1}y dx$  by
\eqref{eq7}. From
Lemma~\ref{lem1}($b$) we obtain $x^{2r+1}y dx=d\left(y
\overline{Q}_{r}\right)+\left(y\overline{q}_{r}\right)
dH$,  thus
$$
\omega=d\left(\int g(x)dx +y\left(
\sum_{r=0}^{\left[\frac{m-2}{2}\right]}a_{2r+2 }
\overline{Q}_{r}\right)\right)+\left(y\left(\sum_{r=0}^{\left[\frac{m-2}{2
} \right]}a_{2r+2}\overline{q}_{r}\right)\right)
dH=dQ+\left(y\overline{q}\right)
dH,
$$
where
$\overline{Q}_{r},\overline{q}_{r}, \overline{q}
\in\R{\left[x^2,H\right]}$ are
homogeneous and
$
\deg_2{\overline{q}}
=\left[\frac{m-2}{2}\right]$. Moreover, a simple computation shows that
\begin{equation}
 \label{eq5}
\tag{$5$}
\overline{q}_{r}=2\sum_{i=0}^{r} {r+1 \choose
i}\left(\frac{r+1-i}{2i
+1}\right)\left(2H\right)^{r-i}\left(x^2-2H\right)^i.
\end{equation}

As
$
\left(y\overline{q}\right)\omega=\overline{q}\hat{g}{
\left(x^2\right)} ydx
+\overline{q}\tilde{g}{\left(x^2\right)}xy dx+
\overline{q}\tilde{f}{\left(x^2\right)}xy^2 dx$
and  $\overline{q}\tilde{f}{\left(x^2\right)}xy^2dx
=\overline{q}\tilde{f}{\left(x^2\right)}
x\left(2H-x^2\right)dx$, it follows that
$\left(y\overline{q}\right)\omega=\overline{q}\hat{g}{
\left(x^2\right)} ydx+dQ_2+q_2dH$ because of
statements $(a)$ and $(b)$ of Lemma \ref{lem1}. Hence we obtain
$$
\int_{\gamma_c} \left(y\bar{q}\right)\omega=
\int_{\gamma_c}\overline{q}
\hat{g}{\left(x^2\right)}ydx=\int_{\gamma_c}
\left(\sum_{r=0}^{\left[(m-2)/2
 \right]}a_{2r+2}\overline{q}_{r}\right)\left(
\sum_{s=0}^{\left[n/2\right]}b_{2s}x^{2s}\right)y dx.
$$
By using  expression \eqref{eq5} of $\overline{q}_{r}$, a straightforward
computation,
and Lemma \ref{lem1}($c$) we obtain the formula given in the statement. Finally,
statement ($c$) follows from the formula given in  statement ($b$).
\end{proof}

\begin{remark}\rm System \eqref{eq1}
with $\mu=\nu=1$, $F_1(x)=-x^2$, and $g_1(x)=1-x^2$
does not satisfy the hypothesis
in definition of
$\tilde{\Hh}_{\nu}^{\mu}(m,n)$ because $g_1(x)$ is not an odd function.
Here $m=n=2$ and from
Theorem~\ref{mth1}($a$) it follows that $\tilde{\Hh}_{1}^{1}(2,2)=0$; however, for
$\ep$ small enough, this system has one
medium amplitude limit cycle. Indeed, we need only to prove that the first non-vanishing coefficient of the displacement function \eqref{df}, associated to the system, has a simple positive zero. The system can be written in the form \eqref{Ll} as $dH-\ep\omega=0$ with $\omega=(1-x^2-2xy)dx$. By Lemma~\ref{lem4}$(a)$, $L_1(c)\equiv0$, and by Theorem~\ref{PPMF} and Lemma~\ref{lem4}$(b)$, $L_2(c)=-\pi c(4-2c)$. Now, system \eqref{eq1}
with $\mu=\nu=2$,
$F_1(x)=-3x^2$, $F_2(x)=-2x^3$, $g_1(x)=x^2+x^3$, and $g_2(x)=\left(-5+25x^2\right)/6$ does not
satisfy the hypothesis in definition of
$\bar{\Hh}_{\nu}^{\mu}(m,n)$ because $F_2(x)$ is not an even function. In this case $m=n=3$ and by
Theorem~\ref{mth1}($b$),
$\bar{\Hh}_{2}^{2}(3,3)=1$; however, for $\ep$ small enough, the resulting system has two
medium amplitude limit cycles. Indeed, following previous ideas, and using Theorem~\ref{PPMF} and Lemma~\ref{lem4} it is easy to see that $L_1(c)\equiv0$, $L_2(c)\equiv0$, and $L_3(c)=-\pi c (c-1)(c-2)$.
\end{remark}

\section{Proof of the main results}
\label{sec4}
We can assume, after a linear change of variables if necessary, that
$F_i(0)=0$ for all $1\leq i \leq \mu$. Suppose that
$F_i(x)=\sum_{r=1}^{m}(a_{i(r-1)}/r)x^r$ and
$g_i(x)=\sum_{s=0}^{n}b_{is}
x^s$. Thus,
$f_i(x)=F_i^{\prime}{(x)}=\sum_{r=0}^{m-1}a_{ir}x^r$ and $g_i(x)$ can be written as
$f_i(x)=\hat{f}_{i}{\left(x^2\right)} +
x \tilde{f}_{i}{\left(x^2\right)}$ and
$g_i(x)=\hat{g}_{i}{\left(x^2\right)} + x
\tilde{g}_{i}{\left(x^2\right)}$, respectively, according to
\eqref{eq7}.

\begin{proof}[Proof of  Theorem \ref{mth1}]
($a$). By hypothesis, $g_i(x)$ is odd for $1\leq i\leq \nu$, which means
that $g_i(x)=x\tilde{g}_{i}{\left(x^2\right)}$  for $1\leq i\leq \nu$.
Let $L_{k}(c)$ be the first
non-vanishing  Poincar\'e--Pontryagin--Melnikov function in \eqref{df}.
If
$k=1$, then the theorem is true. Indeed, we have $
L_1(c)=\int_{\gamma_c} \omega_1=\int_{\gamma_c}
x\tilde{g}_{i}{\left(x^2\right)}  dx+\int_{\gamma_c}
\hat{f}_{1}{\left(x^2\right)} y dx+\int_{\gamma_c}
\tilde{f}_{1}{\left(x^2\right)} xy dx$, and as $\int_{\gamma_c}
x\tilde{g}_{i}{\left(x^2\right)}  dx\equiv 0$, and $\int_{\gamma_c}
\tilde{f}_{1}{\left(x^2\right)} xy dx\equiv 0$ by Corollary \ref{cor1}, we
obtain $L_1(c)=\int_{\gamma_c} \hat{f}_{1}{\left(x^2\right)} y dx$. From
\eqref{eq7} we have
$\deg_2 \hat{f}_{1}{\left(x^2\right)}=\left[(m-1)/2\right]$; hence
$L_1(c)$ has at most $\left[(m-1)/2\right]$ positive zeros  because of Corollary
\ref{cor2}. Moreover,  we can
choose suitable coefficients  of $F_1(x)$ in such a way that $L_1(c)$ has
exactly
$\left[(m-1)/2\right]$ simple positive zeros. Therefore,
by applying  the
Poincar\'e--Pontryagin--Andronov criterion it follows that
$\tilde{\Hh}_{\nu}^{\mu}(m,n)=\left[(m-1)/2\right]$.

Suppose then that $k\geq 2$ and we are therefore in the hypothesis of
Theorem~\ref{th3}.
If $\Omega_l \in \mathcal{A}$ for $1\leq l \leq
k-1$,
then $L_{k}(c)=\int_{\gamma_c}
\omega_{k}$ by Lemma
\ref{lem3}, and
by applying the same idea as in previous paragraph, we obtain
$\tilde{\Hh}_{\nu}^{\mu}(m,n)=\left[(m-1)/2\right]$. Accordingly, it remains to prove
that
$\Omega_l \in \mathcal{A}$ for $1\leq l \leq k-1$.

We proceed by induction on $k$. If $k=2$, then $L_1(c)\equiv 0$, which
implies that
$\Omega_1=\left( x\tilde{g}_{1}{\left(x^2\right)}+xy
\tilde{f}_{1}{\left(x^2\right)} \right) dx \in \mathcal{A}$. We now assume
that the
assertion is true for  $k-2$, and we will prove it for $k-1$. By induction
hypothesis, $\Omega_i \in \mathcal{A}$ for $1 \leq i \leq k-2$, which implies
that $\Omega_i =dQ_i+q_idH$ with $q_i \in \mathcal{S}$  for $1\leq i\leq k-2$ by
Corollary~\ref{cor1}. Furthermore, by
Lemma~\ref{lem3}, $\omega_j \in \mathcal{A}$ for $1\leq j\leq k-2$. Hence
$\overline{\Omega}_{k-1}:=\sum_{i+j=k-1}q_i \omega_j$
with $1 \leq i,\, j\leq k-2$ is an element of $\mathcal{A}$ because of Lemma
\ref{lem2}. Since
$\Omega_{k-1}=\omega_{k-1}+\overline{\Omega}_{k-1}$,
$L_{k-1}(c)=\int_{\gamma_c} \Omega_{k-1} =\int_{\gamma_c}
\omega_{k-1}=\int_{\gamma_c}
\hat{f}_{k-1}{\left(x^2\right)} ydx\equiv 0$, whence $\omega_{k-1}=\left(
x\tilde{g}_{k-1}{\left(x^2\right)}+xy
\tilde{f}_{k-1}{\left(x^2\right)} \right) dx \in \mathcal{A}$. Therefore
$\Omega_{k-1} \in \mathcal{A}$, which completes the proof of statement ($a$).

($b$). First, we will note two properties
concerning $\omega_i$ and $\int_{\gamma_c}\omega_i$ which we will use along
the
proof. Second, we will split the proof into two cases: $m$ odd and $m$ even.

For $1\leq i < \mu_0$ the 1-form $\omega_i=g_i(x)dx$ is exact, that is,
$\omega_i=dQ_i+q_idH$ with $q_i \equiv 0$. Hence, by
Theorem~\ref{th3},  $\Omega_i =\omega_i$ and
$L_i(c)=\int_{\gamma_c}\Omega_i \equiv 0$  for $1\leq i < \mu_0$, and
$L_{\mu_0}(c)=\int_{\gamma_c}\Omega_{\mu_0}=
\int_{\gamma_c}\omega_{\mu_0}$.
On the other hand, since $F_i(x)$ is even for $\mu_0  < i\leq \mu$,
$f_i(x)=x\tilde{f}_{i}{\left(x^2\right)}$  for $\mu_0 < i\leq \mu$. Thus,
$\omega_i=d\left(\int
g_i(x)dx\right)+x\tilde{f}_{i}{\left(x^2\right)}ydx$ for
$ i> \mu_0 $, and as $x^{2r+1}y
dx=d\left(y
\overline{Q}_{0r}\right)+\left(y\overline{q}_{0r}\right)
dH$ because of Lemma~\ref{lem1}($b$), we conclude that
$\omega_i=d\left(\bar{Q}_i\right)+\left(y\bar{q}_i\right)dH$; of course
$\bar{q}_i\equiv 0$ for $i>\mu$. Therefore $\int_{\gamma_c}\omega_i \equiv 0$
for all $i > \mu_0 $.

{\it Case $m$  odd.} If $m$ is odd, then $\deg F_{\mu_0}(x)=m$ because $F_i(x)$
is an even polynomial
for
$\mu_0 < i \leq \mu$. Since
$F_{\mu_0}^{\prime}(x)=f_{\mu_0}(x)=\hat{f}_{\mu_0}{\left(x^2\right)} +
x \tilde{f}_{\mu_0}{\left(x^2\right)}$ has an even degree,
$\hat{f}_{\mu_0}{\left(x^2\right)} \not \equiv 0$. Hence, from Lemma
\ref{lem4}($a$)
it follows that $L_{\mu_0}(c)=\int_{\gamma_c}\omega_{\mu_0}
=\int_{\gamma_c} \hat{f}_{\mu_0}{\left(x^2\right)} ydx \not \equiv 0$, and
it has at most $\left[(m-1)/2\right]$  positive zeros, counting
multiplicities; moreover,  we can choose
suitable coefficients of $F_{\mu_0}(x)$ in such a way that $L_{\mu_0}(c)$ has
exactly
$\left[(m-1)/2\right]$ simple positive zeros. Therefore by the
Poincar\'e--Pontryagin--Andronov criterion,
$\bar{\Hh}_{\nu}^{\mu}(m,n)=\left[(m-1)/2\right]$.

{\it Case $m$  even.} Let $L_{k}(c)$ be the first
non-vanishing  Poincar\'e--Pontryagin--Melnikov function of \eqref{df}.
If $k=\mu_0$, then
$L_{\mu_0}(c)$  has at most $\left[(m-1)/2\right]$ positive
zeros,
counting
multiplicities, because of Lemma~\ref{lem4}($a$). Since  $m$ is even,
$\left[(m-1)/2\right]\leq
\left[m/2\right]+\left[n/2\right]-1$. Hence $L_{\mu_0}(c)$  has at most
$\left[m/2\right]+\left[n/2\right]-1$ positive
zeros, counting multiplicities.

We claim that if $k\geq
\mu_0+1$, then  $\omega_{1},\ldots,\omega_{k-1-\mu_0} \in
\mathcal{A}$, $\Omega_i=dQ_i+q_idH$ with $q_i \in \mathcal{S}$
for ${\mu_0}\leq i \leq k-1$,
and $L_{k}(c)=\int_{\gamma_c} \left(y\bar{q}_{\mu_0}\right)\omega_{k-\mu_0}=
\int_{\gamma_c}\bar{q}_{\mu_0}\hat{g}_{k-\mu_0}{ \left(x^2\right)} y dx$. By
assuming that
this assertion is true and by applying  Lemma \ref{lem4}($b$) we conclude that
$L_{k}(c)$ has at most $\left[m/2\right]+\left[n/2\right]-1$ positive zeros,
counting
multiplicities; moreover,  we can choose
suitable coefficients of $\bar{q}_{\mu_0}$ and $\hat{g}_{k-\mu_0}{
\left(x^2\right)}$ in such a way that $L_{k}(c)$ has exactly
$\left[m/2\right]+\left[n/2\right]-1$ simple positive zeros. Thus,  by
the Poincar\'e--Pontryagin--Andronov criterion,
$\bar{\Hh}_{\nu}^{\mu}(m,n)=\left[m/2\right]+\left[n/2\right]-1$.   Therefore,
to finish the proof of statement ($b$) we need only to confirm the
assertion, which we prove next by proceeding by induction on $k$.

If $k=\mu_0+1$, then we will prove that $\Omega_{\mu_0}=dQ_{\mu_0}
+q_{\mu_0} dH$ with $q_{\mu_0}\in \mathcal{S}$, and that
$L_{\mu_0+1}(c)=\int_{\gamma_c} \bar{q}_{\mu_0}\hat{g}_1{
\left(x^2\right)} ydx$. We know that $\Omega_{\mu_0}=\omega_{\mu_0}$, and from
Lemma
\ref{lem4}($b$) it follows that $\Omega_{\mu_0}=\omega_{\mu_0}=dQ_{\mu_0}
+q_{\mu_0} dH$, where $q_{\mu_0}=y\bar{q}_{\mu_0} \not
\equiv 0\in \mathcal{S}$. On the other hand, by
Theorem~\ref{th3},
$L_{\mu_0+1}(c)=\int_{\gamma_c}\Omega_{\mu_0+1}$, where
$\Omega_{\mu_0+1}=\omega_{\mu_0+1}+q_1\omega_{\mu_0}+\cdots
+q_{\mu_0-1}\omega_{2}
+q_{\mu_0}\omega_{1}$. Since $q_i \equiv 0$ for $1\leq i <
\mu_0$, $\Omega_{\mu_0+1}=\omega_{\mu_0+1}+q_{\mu_0}\omega_{1}$. Moreover,
since
$\int_{\gamma_c}\omega_{\mu_0+1}\equiv 0$,
$L_{\mu_0+1}(c)=\int_{\gamma_c}
\left(y\bar{q}_{\mu_0}\right)\omega_1=\int_{\gamma_c} \bar{q}_{\mu_0}\hat{g}_1{
\left(x^2\right)} ydx$.

If $k=\mu_0+2$, then $
L_{\mu_0+1}(c)=\int_{\gamma_c}
\left(y\bar{q}_{\mu_0}\right)\omega_1=\int_{\gamma_c}
\bar{q}_{\mu_0}\hat{g}_1{ \left(x^2\right)} y dx\equiv 0$. Since
$\bar{q}_{\mu_0}\not \equiv 0$, $\hat{g}_1{ \left(x^2\right)}\equiv 0$ by
Lemma~\ref{lem4}($c$). This implies that
$\Omega_1=\omega_1\in \mathcal{A}$, and by Corollary \ref{cor1},
$\Omega_1=dQ_1+q_1dH$ with $q_1 \in \mathcal{S}$.
Moreover, we know that  $\omega_{\mu_0}=dQ_{\mu_0}
+q_{\mu_0} dH$ with $q_{\mu_0}=y\bar{q}_{\mu_0}\in \mathcal{S}$, and
$\omega_{\mu_0+1}=d\left(\bar{Q}_{\mu_0+1}\right)+\left(y\bar{q}_{\mu_0+1}
\right)dH$. Thus,
$\Omega_{\mu_0+1}=\omega_{\mu_0+1}+q_{\mu_0}
\omega_1=dQ_{\mu_0+1}+q_{\mu_0+1}dH$ with $q_{\mu_0+1} \in \mathcal{S}$
because
of Corollary \ref{cor1}.  On the other hand, from
Theorem \ref{th3}
we have
$$
L_{\mu_0+2}(c)=\int_{\gamma_c}\omega_{\mu_0+2}+
\int_{\gamma_c}q_1\omega_{\mu_0+1}+\int_{\gamma_c}q_2\omega_{\mu_0} +\cdots+
\int_{\gamma_c}q_{\mu_0}\omega_{2}+\int_{\gamma_c}q_{\mu_0+1}\omega_{1}.
$$
As $\omega_1 \in
\mathcal{A}$ and $q_{\mu_0+1} \in \mathcal{S}$, then we have
 $q_{\mu_0+1}\omega_{1} \in \mathcal{A}$ following Lemma
\ref{lem2} and $\int_{\gamma_c}q_{\mu_0+1}\omega_{1}\equiv 0$ by Corollary
\ref{cor1}. In addition, we know that $q_i \equiv 0$ for $1\leq i <
\mu_0$ and $\int_{\gamma_c}\omega_{\mu_0+2}\equiv 0$.  Hence $L_{\mu_0+2}(c)=
\int_{\gamma_c}\left(y\bar{q}_{\mu_0}\right)\omega_{2}=\int_{\gamma_c}\bar{q}_{
\mu_0}\hat{g}_{2}{ \left(x^2\right)} y dx$.

We now assume that the assertion holds for  $k-1$, and we will prove it for $k$.
By Theorem \ref{th3}, $
L_{k}(c)=\int_{\gamma_c}\Omega_{k}$, where
$$
\Omega_{k}=\omega_{k}+q_1\omega_{k-1}+\cdots
+q_{\mu_0-1}\omega_{k+1-\mu_0}
+q_{\mu_0}\omega_{k-\mu_0}
+q_{\mu_0+1}\omega_{k-1-\mu_0}
+\cdots+
q_{k-2}\omega_{2}+
q_{k-1}\omega_{1}.
$$
Since $q_i \equiv 0$ for $1\leq i <
\mu_0$, $\Omega_{k}=
\omega_{k}+q_{\mu_0}\omega_{k-\mu_0}+q_{\mu_0+1}\omega_{k-1-\mu_0} +\cdots+
q_{k-2}\omega_{2}+q_{k-1}\omega_{1}$.

On the other hand, from the induction hypothesis it follows that
$\omega_1,\ldots,\omega_{k-2-\mu_0} \in
\mathcal{A}$, $\Omega_i=dQ_i+q_idH$ with $q_i \in \mathcal{S}$
for $\mu_0\leq i \leq k-2$, and
$L_{k-1}(c)=\int_{\gamma_c}\bar{q}_{\mu_0}\hat{g}_{k-1-\mu_0}{\left(x^2\right)}
y
dx$. Since $L_{k-1}(c) \equiv 0$,
$\hat{g}_{k-1-\mu_0}{\left(x^2\right)}  \equiv 0$ because of
Lemma~\ref{lem4}($c$),
which implies that $\omega_{k-1-\mu_0}
\in \mathcal{A}$. Therefore, $q_{\mu_0}\omega_{k-1-\mu_0} +\cdots+
q_{k-3}\omega_{2}+q_{k-2}\omega_{1} \in \mathcal{A}$ by Lemma~\ref{lem2}. Moreover,
we have $\omega_{k-1}=d\left(\bar{Q}_{k-1}\right)+\left(y\bar{q}_{k-1}
\right)dH$, and by applying Corollary \ref{cor1} we obtain
$$
\Omega_{k-1}=
\omega_{k-1}+q_{\mu_0}\omega_{k-1-\mu_0} +\cdots+
q_{k-3}\omega_{2}+q_{k-2}\omega_{1}=dQ_{k-1}+q_{k-1}dH,\quad
\mbox{with $q_{k-1} \in \mathcal{S}$.}
$$
Hence $q_{\mu_0+1}\omega_{k-1-\mu_0} +\cdots+
q_{k-2}\omega_{2}+q_{k-1}\omega_{1} \in \mathcal{A}$ by Lemma~\ref{lem2}. In
addition, $\omega_{k}=d\left(\bar{Q}_{k}\right)+\left(y\bar{q}_{k}
\right)dH$. Thus, we obtain $
L_{k}(c)=\int_{\gamma_c}q_{\mu_0}\omega_{k-\mu_0}=\int_{\gamma_c}
\left(y\bar{q}_{\mu_0}\right)\omega_{k-\mu_0}=
\int_{\gamma_c}\bar{q}_{\mu_0}\hat{g}_{k-\mu_0}{ \left(x^2\right)} y dx.$
\end{proof}

\subsection*{Acknowledgments}Part of the results of this work come from the
author's postdoctoral stay at  the Departament de Matem\`atiques of the
Universitat Aut\`onoma de
Barcelona.
The author would like to thank the Centre de Recerca Matem\`atica for their
support and hospitality during the period in which this paper was written.

\end{document}